\documentclass[11pt]{article}

\usepackage[T1]{fontenc}
\usepackage[utf8]{inputenc}
\usepackage{lmodern}
\usepackage[a4paper,margin=1in]{geometry}
\usepackage{amsmath,amssymb,amsthm,mathtools}
\usepackage{array}
\usepackage{booktabs}
\usepackage{enumitem}
\usepackage[colorlinks=true,linkcolor=black,citecolor=black,urlcolor=blue]{hyperref}
\usepackage{xcolor}
\usepackage{graphicx}

\newcommand{\Par}{\operatorname{Par}}
\newcommand{\Cl}{\operatorname{Cl}}

\theoremstyle{plain}
\newtheorem{theorem}{Theorem}[section]
\newtheorem{proposition}[theorem]{Proposition}
\newtheorem{lemma}[theorem]{Lemma}
\newtheorem{corollary}[theorem]{Corollary}
\newtheorem{conjecture}[theorem]{Conjecture}
\newtheorem{observation}[theorem]{Observation}
\newtheorem{problem}[theorem]{Problem}

\theoremstyle{definition}
\newtheorem{definition}[theorem]{Definition}

\theoremstyle{remark}
\newtheorem{remark}[theorem]{Remark}

\title{Simplicial shells and thickness in the partition graph}
\author{Fedor B. Lyudogovskiy}
\date{}

\begin{document}

\maketitle

\begin{abstract}
For each positive integer $n$, let $G_n$ be the graph whose vertices are the partitions of
$n$, with edges given by elementary transfers of one unit between parts, followed by
reordering.
In this paper we study the distribution of local simplex dimension in the clique complex
$K_n=\Cl(G_n)$ as a geometric thickness invariant of the partition graph.

For a partition $\lambda\vdash n$, let
\[
\tau_n(\lambda):=\dim_{\mathrm{loc}}(\lambda)
\]
be its simplicial thickness.
This gives rise to the threshold thick zones
\[
T_{\ge r}(n)=\{\lambda:\tau_n(\lambda)\ge r\},
\]
as well as to a shell/core decomposition relative to the boundary framework of $G_n$:
the outer shell $Sh_r(n)$ is the boundary-attached part of $T_{\ge r}(n)$, while the
inner core $Core_r(n)$ is its complementary interior part.

Using local-morphology results established earlier in the series, we work with
simplicial thickness as a local invariant. In the present paper we prove that it is
preserved by conjugation, that the induced thick zones, shells, and cores are
conjugation-invariant, and that the antennas remain strictly one-dimensional in the
simplicial sense and are excluded from all nontrivial thick zones.
The first shell order at which a nontrivial shell can occur is therefore $2$, and the
corresponding shell $Sh_2(n)$ is the triangular skin, while the tetrahedral and higher
simplicial regimes form nested higher-order shells inside the triangular regime.

The paper also develops a computational atlas of simplicial thickness for small and medium
values of $n$.
This yields first-occurrence tables for the regimes $T_{\ge r}(n)$ and supports a stable
atlas-based geometric pattern: substantial higher-dimensional thickening is concentrated
not at the front extremes of the graph, but in its rear-central part.

The paper thereby develops a systematic language for the body geometry of the
partition graph, distinguishing thin regions, a triangular skin, higher simplicial thickening,
outer shells, and inner cores.
\end{abstract}


\medskip


\section{Introduction}

The partition graph $G_n$ is the graph whose vertices are the partitions of $n$, with
adjacency given by an elementary transfer of one unit between two parts, followed by
reordering.
Although this graph is defined by a very simple local move, its global geometry is highly
nonuniform.
Some parts of $G_n$ remain essentially linear, some support a persistent triangular
structure, and others exhibit genuinely higher-dimensional simplicial behaviour.
The purpose of the present paper is to introduce a first systematic language for this
phenomenon.

Graphs on integer partitions defined by minimal local moves belong to the broader
combinatorial setting of partition Gray codes and minimal-change generation; see, for
example, Savage \cite{Savage1989}, Rasmussen, Savage, and West \cite{RasmussenSavageWest1995},
and M\"utze \cite{Mutze2023}. For a different graph model on partitions, based on binary-word
encodings and Hamming adjacency, see Bal \cite{Bal2022}.

More precisely, we study the \emph{simplicial thickness} of the partition graph.
This is measured by the local simplex dimension in the clique complex
\[
K_n:=\Cl(G_n),
\]
that is, by the dimension of the largest simplex of $K_n$ passing through a given vertex.
In earlier papers of this series, this invariant was analyzed locally in terms of the
transfer structure of a partition \cite{Lyudogovskiy2026Local}, while the global clique
complex $K_n$ was studied from a homotopy-theoretic point of view
\cite{Lyudogovskiy2026Homotopy}.
The focus here is different.
Rather than studying the global homotopy type of $K_n$, we study the geometric
distribution of low- and high-dimensional simplicial regimes inside the graph $G_n$ itself.

The guiding geometric picture is that the partition graph has not only an outer framework
but also an interior body, and that this body thickens unevenly as $n$ grows.
At the two extreme antennas the graph stays thin.
Further in, one encounters a boundary-attached triangular layer.
Only later do tetrahedral and higher-dimensional local regimes appear, and the available
data suggest that this stronger thickening is not nucleated at the front end of the graph
but in its rear-central part in the descriptive sense of the atlas layout.
The aim is to express this geometric picture in a conservative combinatorial language. This viewpoint continues the geometric and morphological line
initiated in \cite{Lyudogovskiy2026Growing,Lyudogovskiy2026Simplex}.

The starting point is the local simplicial thickness
\[
\tau_n(\lambda):=\dim_{\mathrm{loc}}(\lambda),
\]
defined for every partition $\lambda\vdash n$.
From this invariant we obtain the threshold thick zones
\[
T_{\ge r}(n):=\{\lambda\in\Par(n):\tau_n(\lambda)\ge r\}.
\]
These zones give a natural nested filtration of the graph by increasing simplicial
complexity.
However, they do not by themselves distinguish between boundary-attached and genuinely
interior behaviour.
For this reason, we introduce for each order $r$ an \emph{outer simplicial shell}
$Sh_r(n)$, defined as the boundary-attached part of $T_{\ge r}(n)$ relative to the
boundary framework, and an \emph{inner simplicial core} $Core_r(n)$, defined as the
complementary part of the same threshold zone.

This language allows us to separate several levels of description that should not be
confused.
First, there is the strictly combinatorial thickness filtration given by the sets
$T_{\ge r}(n)$.
Second, there is the geometric shell/core decomposition of those thick zones relative to
the boundary framework.
Third, there are broader interpretive patterns, such as rear-central concentration of
high-dimensional behaviour, which may be strongly supported by computation without yet
being available in full theorem form.
One of the methodological aims of the paper is to keep these levels distinct.

\subsection*{Main results}

The results of the paper have two levels.

First, we introduce a shell language for simplicial thickening in the partition graph.
This includes the threshold thick zones $T_{\ge r}(n)$, the outer shells $Sh_r(n)$, and
the inner cores $Core_r(n)$.
Within this framework, we use from earlier work the fact that simplicial thickness is a
local invariant, and we prove:
\begin{itemize}[leftmargin=1.5em]
\item it is preserved by conjugation;
\item the antennas are strictly one-dimensional in the simplicial sense;
\item the first shell order at which a nontrivial shell can occur is $2$, with corresponding shell $Sh_2(n)$, the triangular skin;
\item higher shells are nested and develop inside the triangular regime.
\end{itemize}

Second, we build a computational atlas of simplicial thickness and determine the first
occurrences of the regimes $T_{\ge r}(n)$ in the computed range.
This atlas supports a stable rear-central thickening pattern for tetrahedral and higher
simplicial behaviour, in the descriptive sense of the chosen coordinate layout.

The paper is organized as follows.
Section~\ref{sec:preliminaries} recalls the necessary background on the partition graph,
local simplex dimension, and the boundary framework.
Section~\ref{sec:shell-language} introduces the shell language and fixes the definitions of
simplicial thickness, threshold thick zones, outer shells, and inner cores.
Section~\ref{sec:one-dimensional} studies the exact one-dimensional regime and the role of
the framework as the outer supporting structure.
Section~\ref{sec:triangular-skin} turns to the triangular skin.
Section~\ref{sec:higher-thickening} studies the tetrahedral and higher simplicial regimes.
Section~\ref{sec:rear-central} is devoted to maximal-thickness loci and the rear-central
thickening pattern.
Section~\ref{sec:atlas} presents the computational atlas and the first-occurrence tables.
The final section collects conclusions, conjectures, and open problems.

Collectively, these results constitute a first dedicated study of the spatial body geometry
of $G_n$.
The aim is not merely to stratify the graph by local simplex dimension, but to describe
where and how the partition graph ceases to be thin.

\section{Preliminaries and recalled notation}
\label{sec:preliminaries}

In this section we briefly recall the objects and pieces of terminology needed in the
sequel.
Only the material used directly in the present paper is included.

\subsection{The partition graph}

Let $\Par(n)$ denote the set of all partitions of $n$.
We write partitions in weakly decreasing form,
\[
\lambda=(\lambda_1,\lambda_2,\dots,\lambda_\ell),\qquad
\lambda_1\ge \lambda_2\ge \cdots \ge \lambda_\ell>0,\qquad
\sum_{i=1}^{\ell}\lambda_i=n.
\]

The \emph{partition graph} $G_n$ is the graph whose vertex set is $\Par(n)$, with two
partitions adjacent if one is obtained from the other by an elementary transfer of one
unit between two parts, followed by reordering.

Equivalently, in Ferrers-diagram language, an edge corresponds to removing one cell from
one row, adding it to another row, and then restoring weakly decreasing order.

We write
\[
K_n:=\Cl(G_n)
\]
for the clique complex of $G_n$. For standard background on integer partitions and Ferrers
diagrams, see Andrews \cite{Andrews1976}.

\begin{lemma}\label{lem:Gn-connected}
For every $n\ge 2$, the partition graph $G_n$ is connected.
\end{lemma}

\begin{proof}
Let
\[
\lambda=(\lambda_1,\dots,\lambda_\ell)\vdash n.
\]
If $\lambda=(n)$, there is nothing to prove.
Otherwise, one has $\ell\ge 2$, so the first part and the last nonzero part are distinct.
Removing one unit from the last nonzero part and adding it to the first part therefore
produces, after reordering, a different partition of $n$ obtained by an admissible
 elementary transfer in the sense defining the edges of $G_n$. The first part strictly
increases under this move, while the total sum remains $n$; equivalently, the number of
cells outside the first row strictly decreases. Repeating this move therefore eventually
reaches $(n)$. Thus every vertex is connected to $(n)$, so $G_n$ is connected.
\end{proof}

\subsection{Local simplex dimension}

For a vertex $\lambda\in\Par(n)$, the \emph{local simplex dimension} of $\lambda$ in $K_n$
is
\[
\dim_{\mathrm{loc}}(\lambda):=
\max\{\dim \sigma:\lambda\in \sigma,\ \sigma\text{ a simplex of }K_n\}.
\]

Thus $\dim_{\mathrm{loc}}(\lambda)=1$ means that $\lambda$ belongs to an edge but to no
triangle, while $\dim_{\mathrm{loc}}(\lambda)\ge 2$ means that $\lambda$ lies in at least
one triangle, and $\dim_{\mathrm{loc}}(\lambda)\ge 3$ means that it lies in at least one
tetrahedron.

In the present paper we use the notation
\[
\tau_n(\lambda):=\dim_{\mathrm{loc}}(\lambda)
\]
and refer to $\tau_n(\lambda)$ as the \emph{simplicial thickness} of $\lambda$.

We shall use the following result from the earlier local-morphology paper.

\begin{proposition}\label{prop:background-locality}
For each $n\ge 1$, the value $\tau_n(\lambda)$ is determined by the ordered local transfer
type of $\lambda$.
In particular, simplicial thickness is a local invariant in the transfer-theoretic sense.
\end{proposition}

We do not repeat the full formal definition of ordered local transfer type here, since only
this locality consequence is used in the present paper. Roughly speaking, it records the
partitions reachable from $\lambda$ by one admissible transfer together with the local
adjacency pattern among those neighbors, in the ordered form used in
\cite{Lyudogovskiy2026Local}.

\subsection{Conjugation symmetry}

If $\lambda=(\lambda_1,\lambda_2,\dots)$ is a partition, its conjugate partition is denoted
by $\lambda'$.

Conjugation acts naturally on Ferrers diagrams by reflection across the main diagonal.
Since elementary transfers are preserved by this operation, conjugation defines an
automorphism of the graph $G_n$.
Consequently, it also induces a simplicial automorphism of $K_n$.

\subsection{Boundary framework}

We now recall the outer supporting structure of the partition graph.

The two extreme vertices
\[
(n)\qquad\text{and}\qquad (1^n)
\]
are called the \emph{antennas} of $G_n$.

The \emph{main chain} is the distinguished path joining the two antennas, namely
\[
(n)\,--\,(n-1,1)\,--\,(n-2,1^2)\,--\,\cdots\,--\,(2,1^{n-2})\,--\,(1^n).
\]

The \emph{left boundary edge} is the family of two-part partitions
\[
(n-k,k),\qquad 1\le k\le \lfloor n/2\rfloor,
\]
forming a distinguished boundary path in $G_n$.

The \emph{right boundary edge} is its conjugate family
\[
(2^k,1^{\,n-2k}),\qquad 1\le k\le \lfloor n/2\rfloor.
\]
For consecutive values of $k$, the partitions $(n-k,k)$ and $(n-k-1,k+1)$ are adjacent by
a single elementary transfer from the first part to the second, so the left boundary edge
forms a path in $G_n$. The right boundary edge is its conjugate and therefore also forms a
path.

\begin{definition}
The \emph{boundary framework} of $G_n$, denoted by $B_n$, is the union of the main chain,
the left boundary edge, and the right boundary edge.
\end{definition}

This framework will serve as the reference boundary relative to which shells and cores are
defined. Its large-scale role was introduced in \cite{Lyudogovskiy2026Growing} and studied
more specifically from the outer and rear point of view in
\cite{Lyudogovskiy2026Boundary}.

\subsection{Axis, spine, and central language}

Two additional pieces of terminology will be used later when discussing the localization of
thick zones.

The \emph{self-conjugate axis} of $G_n$ is the set of self-conjugate partitions of $n$.
It is fixed pointwise by conjugation.

The \emph{spine} is the distinguished axial substructure formed by the self-conjugate axis
together with the chosen short bridges between consecutive self-conjugate vertices.
In the present paper the spine is used only as a geometric reference object: we do not
need its finer structural properties in the formal development of shells and cores. For the
axial and directional formalizations, see
\cite{Lyudogovskiy2026Axial,Lyudogovskiy2026Directional}.

Likewise, the terms \emph{front}, \emph{side}, and \emph{rear} are used only in the
relative geometric sense established earlier in the series:
the antennas represent the front extremes, the boundary edges represent the lateral
boundary, and the region opposite the antennas will be referred to as the rear part of
the graph.

\subsection{Scope of recalled material}

The present paper does not revisit the global homotopy theory of the clique complex $K_n$;
for that theory, see \cite{Lyudogovskiy2026Homotopy}.
What matters for the present study is instead the local and mesoscopic geometry of the
partition graph:
the distribution of vertices of thickness at least $2$, at least $3$, and higher, and the
way these thick zones are attached to the boundary framework or separated from it.

\section{Simplicial thickness, threshold zones, shells, and cores}
\label{sec:shell-language}

In this section we introduce the basic language used in the paper to describe the
``thickness'' of the partition graph $G_n$.
The starting point is the local simplex dimension, which records the largest simplex
of the clique complex $K_n=\Cl(G_n)$ passing through a given vertex.
From this invariant we obtain a nested family of threshold thick zones.
However, these threshold zones should not by themselves be identified with geometric
shells or cores.
To this end, we introduce an explicit positional distinction between the boundary-facing
part of a thick zone and its interior remainder.

\subsection{Local simplicial thickness}

For a vertex $\lambda \in \Par(n)$, define its \emph{simplicial thickness} by
\[
\tau_n(\lambda):=\dim_{\mathrm{loc}}(\lambda),
\]
that is, the maximum dimension of a simplex of $K_n$ containing $\lambda$.

Equivalently, if the largest clique of $G_n$ containing $\lambda$ has size $m+1$, then
\[
\tau_n(\lambda)=m.
\]

Thus:
\begin{itemize}[leftmargin=1.5em]
\item $\tau_n(\lambda)=1$ means that $\lambda$ lies on an edge of $G_n$ but in no triangle;
\item $\tau_n(\lambda)\ge 2$ means that $\lambda$ participates in at least one triangle;
\item $\tau_n(\lambda)\ge 3$ means that $\lambda$ participates in at least one tetrahedron;
\item and so on.
\end{itemize}

We refer to $\tau_n$ as the \emph{local simplicial thickness profile} of $G_n$.

\subsection{Threshold thick zones and exact regimes}

For each integer $r\ge 0$, define the \emph{$r$-th threshold thick zone} by
\[
T_{\ge r}(n):=\{\lambda\in\Par(n): \tau_n(\lambda)\ge r\}.
\]
We also define the \emph{exact $r$-dimensional simplicial regime} by
\[
T_{=r}(n):=\{\lambda\in\Par(n): \tau_n(\lambda)=r\}.
\]

These sets form a nested filtration
\[
\Par(n)=T_{\ge 0}(n)\supseteq T_{\ge 1}(n)\supseteq T_{\ge 2}(n)\supseteq T_{\ge 3}(n)\supseteq \cdots.
\]

\begin{remark}
The sets $T_{\ge r}(n)$ are purely combinatorial and are defined entirely in terms of
local simplex dimension.
They should be viewed as \emph{thickness zones}, not yet as geometric shells.
In particular, a set $T_{\ge r}(n)$ may have several connected components, may contain
both boundary and interior vertices, and need not resemble a shell in any geometric sense.
\end{remark}

\begin{remark}
The exact regimes $T_{=r}(n)$ are useful for statistics and visualization, but they are
not well suited to serve as the primary shell language of the paper.
Indeed, the sets $T_{=r}(n)$ are generally not nested and need not reflect the outer-to-inner
organization of the graph.
For this reason, the main geometric language below is built from the threshold zones
$T_{\ge r}(n)$ rather than from the exact levels $T_{=r}(n)$.
\end{remark}

\subsection{Outer shells and inner cores}

Let $B_n\subseteq G_n$ denote the boundary framework recalled in Section~\ref{sec:preliminaries}.

For each $r\ge 1$, consider the induced subgraph
\[
G_n[T_{\ge r}(n)].
\]

\begin{definition}
The \emph{outer simplicial shell of order $r$}, denoted
\[
Sh_r(n)\subseteq \Par(n),
\]
is the set of vertices belonging to connected components of $G_n[T_{\ge r}(n)]$
that meet the boundary framework $B_n$.
\end{definition}

\begin{definition}
The \emph{inner simplicial core of order $r$}, denoted
\[
Core_r(n)\subseteq \Par(n),
\]
is the complementary subset
\[
Core_r(n):=T_{\ge r}(n)\setminus Sh_r(n).
\]
Equivalently, it is the set of vertices belonging to connected components of
$G_n[T_{\ge r}(n)]$ that do not meet $B_n$.
\end{definition}

Thus each threshold zone splits canonically into two disjoint parts:
\[
T_{\ge r}(n)=Sh_r(n)\sqcup Core_r(n).
\]

\begin{remark}
The terminology ``shell'' is justified here in a conservative combinatorial sense:
$Sh_r(n)$ is not assumed to be a topological sphere, a manifold-like layer, or even connected.
It is simply the boundary-attached part of the $r$-th thick zone.
Likewise, $Core_r(n)$ is not assumed to be connected or unique.
\end{remark}

\begin{remark}
The notions $Sh_r(n)$ and $Core_r(n)$ are defined relative to the chosen boundary
framework $B_n$.
They are therefore not absolute geometric objects independent of that choice, although
in the present paper the framework is fixed canonically.
\end{remark}

We use the following symmetry statement throughout the paper.

\begin{proposition}\label{prop:conjugation-symmetry}
For every $n\ge 1$ and every partition $\lambda\vdash n$,
\[
\tau_n(\lambda)=\tau_n(\lambda').
\]
Consequently, for every $r\ge 0$, the sets $T_{\ge r}(n)$ and $T_{=r}(n)$ are invariant
under conjugation.
Moreover, for every $r\ge 1$, the shells $Sh_r(n)$ and the cores $Core_r(n)$ are also
conjugation-invariant.
\end{proposition}

\begin{proof}
Conjugation is an automorphism of the partition graph $G_n$, hence a simplicial
automorphism of the clique complex $K_n$.
Therefore it preserves local simplex dimension, so
\[
\tau_n(\lambda)=\tau_n(\lambda').
\]
The conjugation-invariance of the threshold zones and exact regimes follows immediately.

Since the boundary framework $B_n$ is conjugation-invariant, conjugation preserves the
property of a connected component of $G_n[T_{\ge r}(n)]$ to meet $B_n$.
Hence it preserves both $Sh_r(n)$ and $Core_r(n)$.
\end{proof}

\subsection{Low-dimensional special cases}

The first two nontrivial threshold levels are especially important.

\begin{definition}
The threshold zone
\[
T_{\ge 2}(n)
\]
is called the \emph{triangular regime} of the partition graph.
Its boundary-attached part
\[
Sh_2(n)
\]
is called the \emph{triangular skin}.
\end{definition}

\begin{definition}
The threshold zone
\[
T_{\ge 3}(n)
\]
is called the \emph{tetrahedral regime} of the partition graph.
Its boundary-attached part $Sh_3(n)$ is the \emph{outer tetrahedral shell}, and its
interior part $Core_3(n)$ is the \emph{inner tetrahedral core}.
\end{definition}

For $r\ge 4$, the zone $T_{\ge r}(n)$ will be called the \emph{higher simplicial regime
of order $r$}, with associated shell $Sh_r(n)$ and core $Core_r(n)$.

\subsection{Why this choice of language}

There are several conceivable ways to formalize shell geometry in the partition graph.

A first possibility would be to identify the shell of order $r$ with the exact level $T_{=r}(n)$.
This is too rigid for the present paper: exact levels are useful descriptively, but they do not
capture the monotone inward thickening of the graph.

A second possibility would be to call the full threshold zone $T_{\ge r}(n)$ the shell of order $r$.
This is too broad: it mixes outer and inner behaviour and obscures the distinction between a
boundary skin and a genuinely interior body.

We therefore adopt a third, intermediate choice:
\begin{itemize}[leftmargin=1.5em]
\item $T_{\ge r}(n)$ is the formal $r$-thick zone;
\item $Sh_r(n)$ is its outer shell;
\item $Core_r(n)$ is its inner core.
\end{itemize}

This choice preserves the combinatorial filtration by local simplex dimension, while still
allowing a geometric discussion of outer skin, inner structure, and rear-central nucleation
of thickness.

\section{The exact one-dimensional regime and the boundary framework}
\label{sec:one-dimensional}

We now turn to the lowest simplicial level.
At first sight, the phrase ``one-dimensional regime'' may suggest the threshold zone
$T_{\ge 1}(n)$.
However, for $n\ge 2$ this threshold is in fact trivial: every vertex of $G_n$ belongs to
at least one edge.
Thus the genuinely informative object at the lowest level is the \emph{exact}
one-dimensional regime
\[
T_{=1}(n)=T_{\ge 1}(n)\setminus T_{\ge 2}(n),
\]
consisting of vertices that belong to edges but to no triangles.

This section has two aims.
First, we isolate the strictly one-dimensional behaviour of the antennas.
Second, we explain why the boundary framework serves as the natural outer supporting
structure of the graph, while not being identified with the exact one-dimensional regime.

\subsection{The first threshold level}

\begin{proposition}\label{prop:Tge1-all}
For every $n\ge 2$,
\[
T_{\ge 1}(n)=\Par(n).
\]
For $n=1$, one has
\[
T_{\ge 1}(1)=\varnothing.
\]
\end{proposition}

\begin{proof}
For $n=1$, the graph $G_1$ consists of a single isolated vertex, so no vertex belongs to
an edge.

Let now $n\ge 2$.
By Lemma~\ref{lem:Gn-connected}, the graph $G_n$ is connected.
Since $|\Par(n)|\ge 2$, no vertex of a connected graph on more than one vertex can be
isolated.
Hence every partition belongs to at least one edge of $G_n$, so $\tau_n(\lambda)\ge 1$
for every $\lambda\vdash n$.
\end{proof}

\begin{corollary}\label{cor:T=1-complement}
For every $n\ge 2$,
\[
T_{=1}(n)=\Par(n)\setminus T_{\ge 2}(n).
\]
Thus, for $n\ge 2$, the exact one-dimensional regime is precisely the complement of the triangular regime within $\Par(n)$.
\end{corollary}

\begin{proof}
Immediate from Proposition~\ref{prop:Tge1-all}.
\end{proof}

Proposition~\ref{prop:Tge1-all} shows that the first threshold level carries essentially no
internal geometry once $n\ge 2$.
The first meaningful low-dimensional distinction therefore appears between the exact regime
$T_{=1}(n)$ and the triangular regime $T_{\ge 2}(n)$.

\subsection{The antennas}

We next isolate the strongest strictly one-dimensional vertices.

\begin{proposition}\label{prop:antennas-degree1}
For every $n\ge 2$, the antennas $(n)$ and $(1^n)$ have degree $1$ in $G_n$.
\end{proposition}

\begin{proof}
The partition $(n)$ has a unique admissible elementary transfer, namely to
\[
(n-1,1).
\]
Hence $(n)$ has degree $1$.
The statement for $(1^n)$ follows by conjugation symmetry.
\end{proof}

\begin{lemma}\label{lem:deg1-implies-tau1}
Let $v$ be a degree-one vertex of a graph $G$.
Then the local simplex dimension of $v$ in the clique complex $\Cl(G)$ equals $1$.
\end{lemma}

\begin{proof}
A degree-one vertex belongs to an edge but to no triangle.
Hence the largest simplex of $\Cl(G)$ containing it is $1$-dimensional.
\end{proof}

\begin{proposition}\label{prop:antennas-tau1}
For every $n\ge 2$,
\[
\tau_n((n))=\tau_n((1^n))=1.
\]
In particular, for every $r\ge 2$,
\[
(n),(1^n)\notin T_{\ge r}(n).
\]
\end{proposition}

\begin{proof}
Combine Proposition~\ref{prop:antennas-degree1} with Lemma~\ref{lem:deg1-implies-tau1}.
\end{proof}

Thus the two front extremes of the partition graph remain strictly one-dimensional in the
simplicial sense.
No nontrivial thickening begins at the antennas.

\subsection{The boundary framework}

We now return to the boundary framework $B_n$ recalled in
Section~\ref{sec:preliminaries}.
Although $B_n$ is not defined in terms of simplicial thickness, it provides the geometric
boundary relative to which shells and cores are measured.

\begin{proposition}\label{prop:framework-basic}
For every $n\ge 2$, the boundary framework $B_n$ is a connected conjugation-invariant
subgraph of $G_n$ containing both antennas.
\end{proposition}

\begin{proof}
By construction, the main chain is a path joining the two antennas.
The left boundary edge contains $(n-1,1)$, while the right boundary edge contains
$(2,1^{n-2})$; both of these vertices lie on the main chain. Hence each boundary edge
meets the main chain and lies in the same connected component as the main chain. Their
union with the main chain is therefore connected. Conjugation invariance is immediate from
the symmetric construction, and the union contains both antennas.
\end{proof}

\begin{remark}\label{rem:framework-not-equals-T1}
The boundary framework $B_n$ should not be identified with the exact one-dimensional regime
$T_{=1}(n)$.
The framework is a distinguished geometric carrier of the graph, whereas $T_{=1}(n)$ is a
thickness-defined locus.
Some framework vertices may belong to triangles or higher simplices, so the two notions
need not coincide.
\end{remark}

\begin{remark}\label{rem:framework-reference-boundary}
The boundary framework is the reference boundary used in the shell/core formalism.
Outer shells are defined relative to their attachment to $B_n$, even though $B_n$ is not
required to be contained in any fixed threshold zone $T_{\ge r}(n)$.
\end{remark}

\subsection{Interpretation}

The conclusions of this section can be summarized as follows.

For $n\ge 2$, the threshold condition $\tau_n(\lambda)\ge 1$ is automatic, so the first
meaningful low-dimensional distinction is between the exact one-dimensional regime
$T_{=1}(n)$ and the triangular regime $T_{\ge 2}(n)$.

At the very front ends of the graph, the situation is completely rigid: the antennas have
simplicial thickness exactly $1$ and are excluded from all higher thick zones.
At the same time, the boundary framework provides the natural outer carrier relative to
which higher shells are attached.
This makes it the natural geometric reference object for the study of triangular skin,
tetrahedral regime, and subsequent thickening.

\section{The triangular skin}
\label{sec:triangular-skin}

We now pass to the first genuinely nontrivial level of simplicial thickening.
Since $T_{\ge 1}(n)=\Par(n)$ for $n\ge 2$, the first informative shell is the shell of
order $2$, namely the boundary-attached part of the triangular regime.
This is the object that we call the \emph{triangular skin} of $G_n$.

\subsection{The first nontrivial shell order}

\begin{proposition}\label{prop:shell1-trivial}
For every $n\ge 2$,
\[
Sh_1(n)=\Par(n),
\qquad
Core_1(n)=\varnothing.
\]
Hence order $2$ is the first shell order at which a nontrivial shell may occur. The
corresponding shell
\[
Sh_2(n),
\]
is the triangular skin.
\end{proposition}

\begin{proof}
By Proposition~\ref{prop:Tge1-all}, one has
\[
T_{\ge 1}(n)=\Par(n).
\]
Thus the induced subgraph $G_n[T_{\ge 1}(n)]$ is just $G_n$ itself.
Since $G_n$ is connected and contains the boundary framework $B_n$, every vertex belongs to
a connected component of $G_n[T_{\ge1}(n)]$ meeting $B_n$.
Therefore
\[
Sh_1(n)=\Par(n).
\]
By definition,
\[
Core_1(n)=T_{\ge1}(n)\setminus Sh_1(n)=\varnothing.
\]
\end{proof}

\begin{definition}
The shell
\[
Sh_2(n)
\]
is called the \emph{triangular skin} of the partition graph $G_n$.
\end{definition}

Thus the triangular skin is defined as the shell of order $2$, which is the first shell
order at which a nontrivial shell may occur.

\subsection{Basic structural properties}

We next record the elementary structural properties inherited from the general theory of
shells and thickness zones.

\begin{corollary}\label{cor:triangular-skin-conjugation}
For every $n\ge 1$, the triangular skin $Sh_2(n)$ and the complementary core $Core_2(n)$
are invariant under conjugation.
\end{corollary}

\begin{proof}
This is the case $r=2$ of Proposition~\ref{prop:conjugation-symmetry}.
\end{proof}

\begin{corollary}\label{cor:triangular-skin-no-antennas}
For every $n\ge 2$,
\[
(n),(1^n)\notin Sh_2(n)\cup Core_2(n).
\]
\end{corollary}

\begin{proof}
By Proposition~\ref{prop:antennas-tau1}, the antennas do not belong to $T_{\ge2}(n)$.
Since
\[
Sh_2(n)\cup Core_2(n)=T_{\ge2}(n),
\]
the claim follows.
\end{proof}

\begin{corollary}\label{cor:skin-meets-framework-away-from-antennas}
If $Sh_2(n)\neq\varnothing$, then every connected component of the triangular skin meets
the boundary framework away from the antennas.
\end{corollary}

\begin{proof}
By definition, every connected component of the triangular skin meets $B_n$.
By Corollary~\ref{cor:triangular-skin-no-antennas}, the antennas themselves do not belong
to $Sh_2(n)$.
\end{proof}

These statements already distinguish the triangular skin from the trivial first shell:
it is boundary-attached, but it does not include the front extremes.

\subsection{The triangular regime versus the exact one-dimensional regime}

Because $T_{\ge 1}(n)$ is trivial for $n\ge 2$, the first meaningful low-dimensional split
is the decomposition
\[
\Par(n)=T_{=1}(n)\sqcup T_{\ge 2}(n).
\]
In this decomposition, $T_{=1}(n)$ is the exact one-dimensional regime, while
$T_{\ge 2}(n)$ is the triangular regime.

Thus the triangular regime $T_{\ge 2}(n)$ is the first threshold zone that distinguishes
genuinely simplicially thick vertices from strictly one-dimensional ones.
Equivalently, by Corollary~\ref{cor:T=1-complement},
\[
T_{=1}(n)=\Par(n)\setminus T_{\ge 2}(n).
\]

The triangular skin $Sh_2(n)$ is therefore the outer part of the first potentially
nontrivial thick zone.
By contrast, $Core_2(n)$ measures any interior triangle-bearing behaviour that is already
detached from the framework.

\subsection{Interpretation}

The preceding results justify the triangular skin as the first geometric shell of the
partition graph in the simplicial sense.

First, it occurs at the first shell order at which a nontrivial shell can arise:
the shell of order $1$ is still the whole graph, whereas the shell of order $2$ isolates
the boundary-attached part of the triangular regime.

Second, it is already separated from the front extremes:
the antennas remain strictly one-dimensional and do not belong to the triangular regime.

Third, it is symmetric under conjugation, and hence naturally compatible with the axial
geometry of the graph.

At this stage, however, we do \emph{not} claim that the triangular skin is always
connected, or that it always forms a single uniform strip along the whole outer boundary.
Those stronger geometric features belong to the computational atlas and to the descriptive
analysis developed later in the paper.
The point established in this section is more modest:
\[
Sh_2(n)
\]
is the boundary-attached shell of order $2$, which is the first shell order at which
a nontrivial shell can occur.

\section{Tetrahedral and higher-dimensional thickening}
\label{sec:higher-thickening}

We now pass from the shell of order $2$ to the genuinely higher simplicial regimes.
The triangular skin $Sh_2(n)$ captures the boundary-attached part of the triangular regime whenever it is nonempty.
Beyond that level, one encounters vertices lying in tetrahedra and, later, in still
higher-dimensional simplices of the clique complex.

The emphasis in this section is structural rather than asymptotic.
We record the formal hierarchy of higher thick zones, show how it interacts with the shell
language, and isolate the tetrahedral regime as the first level at which one can speak
about a genuinely higher-dimensional simplicial body.

\subsection{The tetrahedral and higher simplicial regimes}

\begin{definition}
The threshold zone
\[
T_{\ge 3}(n)=\{\lambda\in\Par(n):\tau_n(\lambda)\ge 3\}
\]
is called the \emph{tetrahedral regime} of the partition graph $G_n$.
Its shell $Sh_3(n)$ will be called the \emph{outer tetrahedral shell}, and its core
$Core_3(n)$ the \emph{inner tetrahedral core}.
\end{definition}

\begin{definition}
For each $r\ge 4$, the threshold zone
\[
T_{\ge r}(n)=\{\lambda\in\Par(n):\tau_n(\lambda)\ge r\}
\]
is called the \emph{higher simplicial regime of order $r$}.
Its associated shell and core are denoted $Sh_r(n)$ and $Core_r(n)$.
\end{definition}

Thus the simplicial-thickness filtration continues beyond the triangular level as
\[
T_{\ge 2}(n)\supseteq T_{\ge 3}(n)\supseteq T_{\ge 4}(n)\supseteq \cdots.
\]

\subsection{Higher thickening occurs inside the triangular regime}

The filtration by threshold thick zones immediately implies that every genuinely
higher-dimensional simplicial region is contained in the triangular regime.

\begin{proposition}\label{prop:higher-inside-triangular}
For every $n\ge 1$ and every $r\ge 3$,
\[
T_{\ge r}(n)\subseteq T_{\ge 2}(n).
\]
Consequently,
\[
Sh_r(n)\subseteq Sh_2(n).
\]
\end{proposition}

\begin{proof}
The first inclusion is immediate from the definition:
if $\tau_n(\lambda)\ge r$ with $r\ge 3$, then in particular $\tau_n(\lambda)\ge 2$.

For the second inclusion, every boundary-attached connected component of
$G_n[T_{\ge r}(n)]$ is contained in a connected component of $G_n[T_{\ge2}(n)]$.
If the smaller component meets the boundary framework, then the larger one does as well.
Hence every vertex of $Sh_r(n)$ belongs to $Sh_2(n)$.
\end{proof}

\begin{corollary}\label{cor:tetrahedral-inside-triangular}
For every $n\ge 1$,
\[
T_{\ge 3}(n)\subseteq T_{\ge 2}(n)
\qquad\text{and}\qquad
Sh_3(n)\subseteq Sh_2(n).
\]
Thus tetrahedral boundary behaviour, when present, occurs inside the triangular skin.
\end{corollary}

\begin{proof}
This is the case $r=3$ of Proposition~\ref{prop:higher-inside-triangular}.
\end{proof}

This expresses a basic geometric principle:
higher simplicial thickening does not replace the triangular layer, but develops inside it.

\subsection{Nested shells and non-nested cores}

The threshold zones are nested by definition.
For the shell language, this has an important one-sided consequence.

\begin{proposition}\label{prop:shells-nested}
For every $n\ge 1$ and every $r\ge 1$,
\[
Sh_{r+1}(n)\subseteq Sh_r(n).
\]
In particular,
\[
Sh_2(n)\supseteq Sh_3(n)\supseteq Sh_4(n)\supseteq \cdots.
\]
\end{proposition}

\begin{proof}
Every boundary-attached connected component of $G_n[T_{\ge r+1}(n)]$ is contained in a
connected component of $G_n[T_{\ge r}(n)]$.
If the smaller component meets the boundary framework, then the larger one does as well.
Hence every vertex of $Sh_{r+1}(n)$ belongs to $Sh_r(n)$.
\end{proof}

\begin{remark}\label{rem:cores-not-automatically-nested}
No analogous monotonicity for the inner cores $Core_r(n)$ is asserted here.
A higher-order interior component may be contained in a boundary-attached component of a
lower-order threshold zone.
\end{remark}

\begin{remark}
The monotonicity in Proposition~\ref{prop:shells-nested} is monotonicity in the thickness
order $r$ for fixed $n$.
It should not be confused with any monotonicity statement in the parameter $n$.
\end{remark}

This distinction is one of the reasons for separating shells from cores rather than working
only with the threshold zones themselves.

\subsection{First-occurrence parameters}

To describe the onset of higher simplicial regimes, we introduce the corresponding
first-occurrence parameters; their concrete values will be determined in Section~\ref{sec:atlas}.

\begin{definition}
For each integer $r\ge 2$, define
\[
n_r:=\min\{n\ge 1:T_{\ge r}(n)\neq\varnothing\},
\]
whenever this set is nonempty.
Equivalently, $n_r$ is the first value of $n$ for which the partition graph $G_n$
contains a vertex of simplicial thickness at least $r$.
\end{definition}

Thus $n_2$ is the first triangular threshold, $n_3$ the first tetrahedral threshold, and
so on.
The concrete values of these parameters, as well as the corresponding geometric atlases,
will be determined in Section~\ref{sec:atlas}.

\begin{remark}\label{rem:first-occurrences-not-here}
The present section does not attempt to prove explicit values of $n_r$.
Its purpose is only to define the hierarchy of higher regimes and to establish the formal
relations between them.
The actual first-occurrence table belongs to the computational part of the paper.
\end{remark}

\subsection{Interpretation}

The tetrahedral regime
\[
T_{\ge 3}(n)
\]
is the first genuinely higher simplicial regime of the partition graph.
Formally, it sits inside the triangular regime, and its shell sits inside the triangular
skin.

This leads to a useful picture.
The triangular skin provides the first nontrivial outer simplicial shell.
Inside that shell, one may then encounter stronger forms of local simplicial thickness:
first tetrahedral, then higher-order.
The nested shell structure
\[
Sh_2(n)\supseteq Sh_3(n)\supseteq Sh_4(n)\supseteq \cdots
\]
captures the boundary-facing part of this process.

At the same time, the corresponding inner cores require more caution.
They are the natural candidates for the genuinely interior body of the graph, but they do
not automatically form a nested sequence.
For that reason, the geometric analysis of rear-central thickening in the next section will
be formulated in terms of localization patterns rather than in terms of a naive monotone
core filtration.

\section{Maximal-thickness loci and rear-central thickening}
\label{sec:rear-central}

The previous sections introduced the shell language and the hierarchy of triangular,
tetrahedral, and higher simplicial regimes.
We now turn to the geometric question that motivated much of this paper:
\emph{where} does the first substantial thickening of the partition graph occur?

The guiding geometric observation is that higher simplicial thickening does not appear
to be nucleated at the antennas, nor uniformly along the outer boundary.
Instead, the available data indicate a persistent concentration of tetrahedral and higher
behaviour in the rear-central part of the graph.
In this section we separate what can be stated formally from what is presently supported
primarily by computation.

\subsection{Maximal-thickness loci}

A natural first object is the set of vertices of maximal simplicial thickness.

\begin{definition}
For each $n\ge 1$, let
\[
\tau_{\max}(n):=\max\{\tau_n(\lambda):\lambda\in\Par(n)\},
\]
and define the \emph{maximal-thickness locus} by
\[
M_n:=\{\lambda\in\Par(n):\tau_n(\lambda)=\tau_{\max}(n)\}.
\]
\end{definition}

Thus $M_n$ consists of the vertices where the simplicial thickness of $G_n$ is largest.
Equivalently,
\[
M_n=T_{=\tau_{\max}(n)}(n).
\]
In particular, if $\tau_{\max}(n)\ge r$, then
\[
M_n\subseteq T_{\ge r}(n).
\]

\begin{proposition}\label{prop:max-locus-conjugation}
For every $n\ge 1$, the maximal-thickness locus $M_n$ is invariant under conjugation.
\end{proposition}

\begin{proof}
Immediate from Proposition~\ref{prop:conjugation-symmetry}.
\end{proof}

\begin{corollary}\label{cor:max-locus-no-antennas}
If $\tau_{\max}(n)\ge 2$, then neither antenna belongs to $M_n$.
\end{corollary}

\begin{proof}
By Proposition~\ref{prop:antennas-tau1}, both antennas have simplicial thickness $1$.
\end{proof}

Thus maximal-thickness loci are automatically separated from the two front extremes once
the graph exhibits any nontrivial thickening.

\subsection{Formal limits of the shell language}

The shell/core formalism is well suited for separating boundary-attached and genuinely
interior behaviour.
However, it does not by itself identify a distinguished \emph{rear-central} region.
That notion involves additional geometric interpretation relative to the overall shape of
the graph.

More precisely, the following facts are strict consequences of the previous sections:
\begin{itemize}[leftmargin=1.5em]
\item higher simplicial regimes are contained in the triangular regime;
\item higher shells are nested inside the triangular skin;
\item antennas are excluded from all nontrivial thick zones;
\item maximal-thickness loci are conjugation-invariant.
\end{itemize}

What these facts do \emph{not} yet determine is whether the first tetrahedral or higher
zones are concentrated near the side boundary, near the self-conjugate axis, around the
spine, or more specifically in the rear-central part of the graph.
That further localization question is addressed computationally below.

\begin{remark}
In this paper, the term ``rear-central'' is used as a descriptive term anchored in the
computational layout rather than as a formally axiomatized subset of $G_n$. Roughly
speaking, in the coordinate layout $(\lambda_1,\ell(\lambda))$ it refers to regions away
from the two antennas and the extreme boundary strips, where the largest part and the
number of parts are comparatively balanced.
\end{remark}

\subsection{Computational rear-central pattern}

The atlas developed in Section~\ref{sec:atlas} supports the following stable finite-range
phenomenon.

\begin{observation}[Computational rear-central pattern]
\label{obs:rear-central-pattern}
In the computed range
\[
7\le n\le 30,
\]
that is, from the first tetrahedral threshold onward, the maximal-thickness loci $M_n$
and the first realizations of tetrahedral and higher simplicial regimes are concentrated in
the rear-central part of the partition graph. This is a finite-range observation,
supported by the computational atlas in Section~\ref{sec:atlas} and drawn from the
computed values of $\tau_n(\lambda)$, the first-occurrence tables, and the representative figures there.
\end{observation}

This restriction to $n\ge 7$ is essential.
For $4\le n\le 6$, the maximal simplicial thickness is still only $2$, so the graph has
not yet entered the tetrahedral regime.

This formulation is deliberately restrained.
It does not claim that every vertex of $T_{\ge 3}(n)$ is rear-central, nor that the
higher-dimensional regimes are eventually connected, unique, or asymptotically confined to
a sharply defined axial neighborhood.
The claim made here is weaker and more robust:
\emph{the onset of substantial thickening is consistently rear-central rather than frontal.}

\subsection{Relation to the axis and the spine}

The rear-central pattern is naturally related to the axial morphology developed earlier in
the series.

Because the maximal-thickness locus is conjugation-invariant, it is compatible with the
self-conjugate symmetry of the graph.
This makes the self-conjugate axis and the spine natural reference objects for describing
the placement of thick zones.
In particular, when a high-thickness region appears in a conjugation-symmetric position,
it is reasonable to compare it with the axis or with neighborhoods of the spine.

However, the present paper does not require a theorem asserting that maximal-thickness
vertices lie \emph{on} the axis or even at uniformly bounded distance from the spine.
The computed examples often suggest such a relationship, but at present this belongs to
the descriptive and conjectural level rather than to the strict theorem layer.

\begin{remark}\label{rem:axis-spine-reference}
For the purposes of this paper, the axis and the spine are used as \emph{reference
geometries} rather than as exact carriers of the thick body.
They help describe the location of thickening, but the shell/core formalism itself does
not reduce to axial language.
\end{remark}

\subsection{A conjectural strengthening}

The computational evidence suggests a stronger statement, which we formulate separately.

\begin{conjecture}\label{conj:rear-central-nucleation}
For all sufficiently large $n$, the first genuinely higher simplicial behaviour of the
partition graph is nucleated in the rear-central part of $G_n$.
Equivalently, the earliest tetrahedral and higher-order thick zones are asymptotically
separated from the front extremes and are organized around the central rear body of the
graph rather than around its antennas.
\end{conjecture}

A stronger variant would assert that the maximal-thickness locus eventually remains
within a bounded geometric neighborhood of the spine, or of the rear-central part of the
self-conjugate axis.
At present we do not formulate this as a principal conjecture, because the current
evidence, though suggestive, is better suited to motivate further computation than to
support a sharper universal claim.

\subsection{Interpretation}

The results of this section clarify the role of rear-central thickening in the paper.

At the strict formal level, one can prove that higher simplicial regimes sit inside the
triangular regime, that higher shells are nested, that antennas are excluded from all
nontrivial thick zones, and that maximal-thickness loci are conjugation-symmetric.

At the computational level, one observes something stronger:
the first substantial thickening is not merely non-antennal, but rear-central.
This is the sense in which the partition graph appears to develop an inner body.
The graph does not simply become thicker everywhere at once.
Rather, higher-dimensional simplicial behaviour emerges in a geometrically biased way,
with the rear-central part of the graph acting as the first stable site of substantial
thickening.

This rear-central nucleation pattern is one of the main geometric messages of the paper.
Its current status, however, is intentionally limited:
it is supported strongly by the computed atlas, but only partially by formal theorem-level
arguments.

\section{Computational atlas and first-occurrence tables}
\label{sec:atlas}

This section contains the finite computational part of the paper.
Its purpose is fourfold:
\begin{itemize}[leftmargin=1.5em]
\item to record the first-occurrence values
\[
n_r:=\min\{n:T_{\ge r}(n)\neq\varnothing\},
\]
for the first simplicial regimes;
\item to present an atlas of simplicial thickness maps for small and medium values of $n$;
\item to describe the placement of the maximal-thickness loci
\[
M_n=\{\lambda:\tau_n(\lambda)=\tau_{\max}(n)\};
\]
\item to provide the explicit finite evidence behind the rear-central thickening pattern
formulated in Section~\ref{sec:rear-central}.
\end{itemize}

Throughout this section, all statements are finite computational statements relative to the
enumeration range under consideration.
In the present version, the computed range is
\[
1\le n\le 30,
\]
so the upper bound is
\[
N=30.
\]

\subsection{Computational setup}

The computed range in the present paper is
\[
1\le n\le 30.
\]
For each such $n$, we enumerate the vertices of the partition graph $G_n$, that is, all
partitions $\lambda\vdash n$. In particular, this is a complete finite computation rather
than a sampled one; the largest graph considered is $G_{30}$, with
\[
|\Par(30)|=5604
\]
vertices.

For each partition $\lambda\vdash n$, we enumerate all admissible elementary transfers
starting from $\lambda$. This determines the full local transfer data of $\lambda$: the
set of admissible moves together with the ordered source/target structure required by
Proposition~\ref{prop:background-locality}. Applying that imported result then yields the
simplicial thickness
\[
\tau_n(\lambda)=\dim_{\mathrm{loc}}(\lambda).
\]
From these values we then derive the associated threshold zones
\[
T_{\ge r}(n)=\{\lambda:\tau_n(\lambda)\ge r\},
\]
their shell/core decompositions, and the maximal-thickness loci
\[
M_n=\{\lambda:\tau_n(\lambda)=\tau_{\max}(n)\}.
\]
Equivalently, the first-occurrence values are read off from the maxima
\[
\tau_{\max}(n)=\max\{\tau_n(\lambda):\lambda\in\Par(n)\}
\]
via the identity
\[
n_r=\min\{n:\tau_{\max}(n)\ge r\}.
\]
No heuristic search or partial sampling is used at this stage.

\begin{proposition}\label{prop:finite-verification}
For each fixed $n$ with $1\le n\le 30$, the procedure described above computes the full
function
\[
\tau_n:\Par(n)\to \mathbb{Z}_{\ge 0}.
\]
Consequently it computes
\[
\tau_{\max}(n)=\max\{\tau_n(\lambda):\lambda\in\Par(n)\}
\]
and decides, for every $r\ge 0$, whether the regime $T_{\ge r}(n)$ is empty.
\end{proposition}

\begin{proof}
For fixed $n\le 30$, the set $\Par(n)$ is finite, and for each $\lambda\in\Par(n)$ the set
of admissible elementary transfers out of $\lambda$ is finite. Complete enumeration of
those transfers therefore determines, for every partition $\lambda\vdash n$, the full local
transfer data needed to recover its ordered local transfer type. By
Proposition~\ref{prop:background-locality}, this in turn determines $\tau_n(\lambda)$ for
every such $\lambda$, and hence determines the maximum value $\tau_{\max}(n)$. The final claim is immediate from the definition
\[
T_{\ge r}(n)=\{\lambda\in\Par(n):\tau_n(\lambda)\ge r\}.
\]
\end{proof}

\begin{remark}
The computation used in this paper is completely finite and exhaustive in the range
$1\le n\le 30$: every partition $\lambda\vdash n$ is processed, every admissible local
transfer out of $\lambda$ is enumerated, and the resulting local transfer data are then
converted into the value $\tau_n(\lambda)$ via Proposition~\ref{prop:background-locality}.
Thus the theorem-level claims of this section rest on the full computed dataset
$\{\tau_n(\lambda):\lambda\vdash n,\ 1\le n\le 30\}$, not on visual inspection of the
atlas figures. The complete dataset is available from the author upon request.
\end{remark}

We organize the resulting data in two complementary ways:
\begin{enumerate}[leftmargin=1.5em]
\item as \emph{first-occurrence tables}, recording the smallest $n$ for which a given
simplicial regime appears;
\item as a \emph{geometric atlas}, consisting of drawings of $G_n$ colored by the thickness
function $\tau_n$ and, where convenient, by the threshold zones $T_{\ge r}(n)$ for small
values of $r$.
\end{enumerate}

\begin{remark}
The role of the atlas is not merely illustrative.
It is part of the evidence for the geometric claims of the paper, especially for the
distinction between boundary-attached triangular behaviour and the later rear-central
nucleation of tetrahedral and higher thickening.
\end{remark}

\subsection{First occurrences of simplicial regimes}

We first record the values of
\[
n_r:=\min\{n:T_{\ge r}(n)\neq\varnothing\}
\]
for the first few nontrivial orders $r$.

\begin{table}[ht]
\centering
\begin{tabular}{c|c|p{8.4cm}}
$r$ & $n_r$ & first observed geometric description \\
\hline
2 & 4  & first nontrivial triangular regime \\
3 & 7  & first tetrahedral regime; rear-central $4$-vertex maximal-thickness locus \\
4 & 11 & first regime of order $4$; rear-central $5$-vertex maximal-thickness locus \\
5 & 16 & first regime of order $5$; rear-central $6$-vertex maximal-thickness locus \\
6 & 22 & first regime of order $6$; rear-central $7$-vertex maximal-thickness locus \\
7 & 29 & first regime of order $7$; rear-central $8$-vertex maximal-thickness locus
\end{tabular}
\caption{First-occurrence values $n_r$ for the regimes $T_{\ge r}(n)$ in the computed range $1\le n\le 30$.}
\label{tab:first-occurrences}
\end{table}

\begin{theorem}\label{thm:first-occurrences}
The first-occurrence values $n_2,n_3,n_4,\dots$ in the computed range are exactly those
listed in Table~\ref{tab:first-occurrences}.
In particular, the triangular, tetrahedral, and higher simplicial regimes appear in
strictly increasing order throughout the computed range.
\end{theorem}

\begin{proof}
By Proposition~\ref{prop:finite-verification}, the complete computation determines the
values $\tau_{\max}(n)$ for every integer $n$ with $1\le n\le 30$. For each fixed order
$r$, the regime $T_{\ge r}(n)$ is nonempty if and only if $\tau_{\max}(n)\ge r$. Hence the
first-occurrence value
\[
n_r=\min\{n:T_{\ge r}(n)\neq\varnothing\}
\]
is exactly the smallest integer $n$ in the computed range with $\tau_{\max}(n)\ge r$.
Applying this criterion to the computed values of $\tau_{\max}(n)$ for $1\le n\le 30$ yields the
entries recorded in Table~\ref{tab:first-occurrences}. The displayed values are strictly
increasing for $r=2,3,4,5,6,7$, which proves the final claim.
\end{proof}

\begin{remark}
The complete finite output of the computation consists of the values $\tau_n(\lambda)$ for
all partitions $\lambda\vdash n$ with $1\le n\le 30$. Tables~\ref{tab:first-occurrences}
and \ref{tab:max-thickness-loci} record only the threshold changes and selected geometric
summaries extracted from that full dataset.
\end{remark}

\subsection{Atlas of simplicial thickness maps}

To visualize the body geometry of $G_n$, we draw the partition graph with vertices colored
according to the value of $\tau_n(\lambda)$.

The atlas is organized by representative values of $n$ chosen to show:
\begin{itemize}[leftmargin=1.5em]
\item the first appearance of the triangular regime;
\item the first appearance of the tetrahedral regime;
\item the subsequent growth of higher-order thick zones;
\item the changing relation between boundary-attached shells and interior cores.
\end{itemize}

A typical atlas page contains:
\begin{enumerate}[leftmargin=1.5em]
\item the full graph $G_n$ colored by $\tau_n$;
\item the triangular skin $Sh_2(n)$ highlighted separately;
\item the tetrahedral regime $T_{\ge 3}(n)$, when nonempty;
\item the maximal-thickness locus $M_n$.
\end{enumerate}

In the figures below, vertices are placed using the coordinates
\[
(\lambda_1,\ell(\lambda)),
\]
that is, largest part versus number of parts, with small offsets added when several
partitions share the same pair of values.
This produces a stable triangular reference layout compatible with conjugation symmetry.
For visual consistency with the framework-oriented drawings used elsewhere in the paper,
smaller values of $\ell(\lambda)$ are placed higher on the page.

\begin{figure}[ht]
\centering
\includegraphics[width=.78\textwidth]{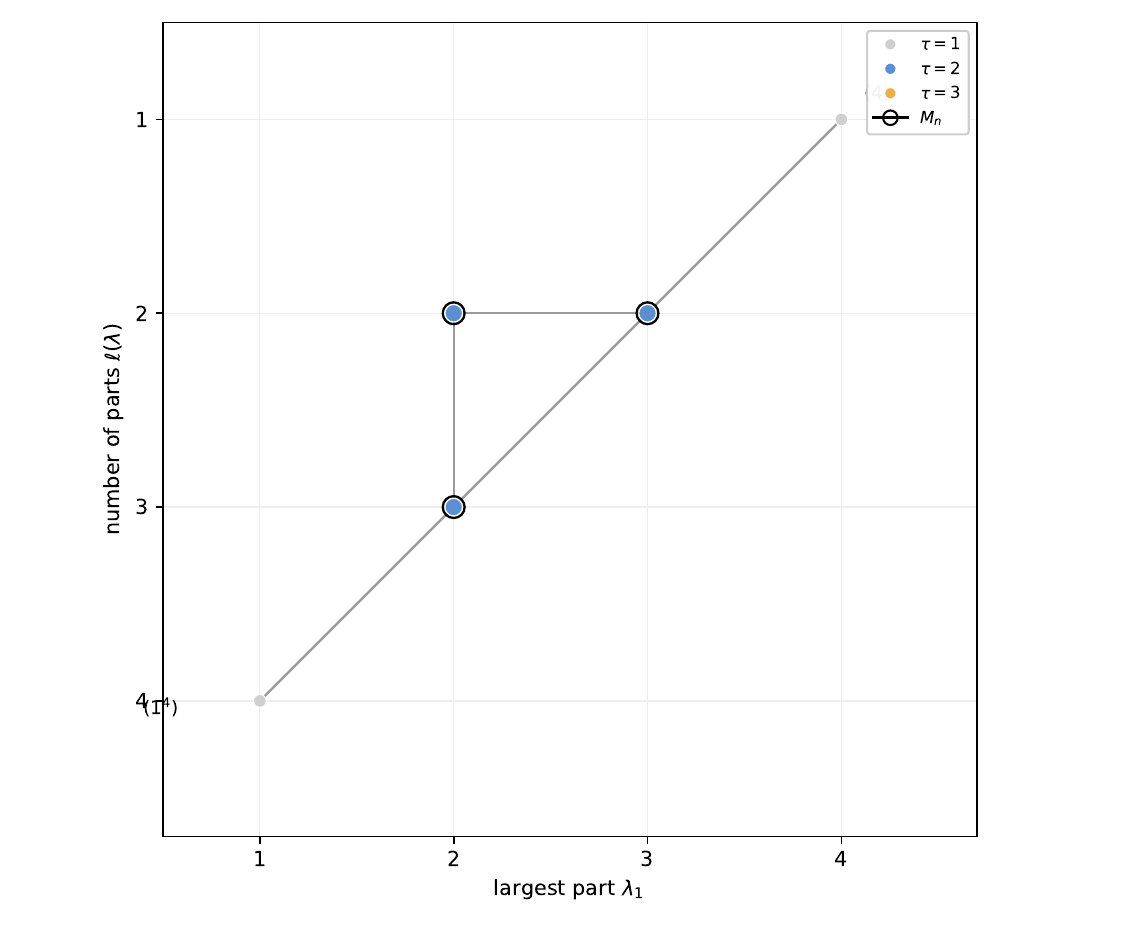}
\caption{Thickness map of $G_4$ at the first triangular threshold $n_2=4$.
Vertices are colored by simplicial thickness $\tau_4(\lambda)$, and the maximal-thickness
locus $M_4$ is indicated by black outlines.
At this stage the only nontrivial thickness value is $\tau=2$, realized by the first
triangular regime.}
\label{fig:atlas-n4}
\end{figure}

\begin{figure}[ht]
\centering
\includegraphics[width=.78\textwidth]{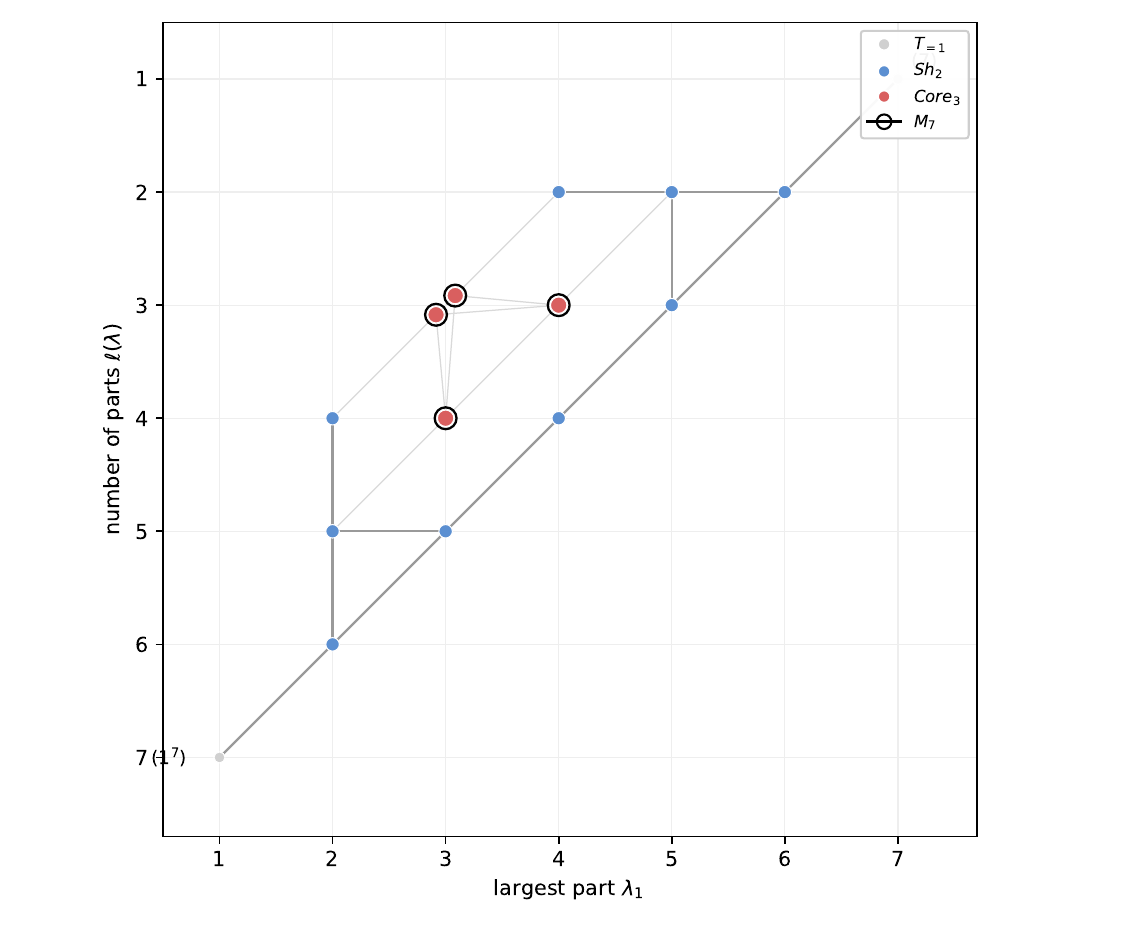}
\caption{Threshold-zone view of $G_7$ at the first tetrahedral threshold $n_3=7$.
Gray vertices form the exact one-dimensional regime $T_{=1}(7)$, blue vertices belong to
the triangular skin $Sh_2(7)$, and red vertices form the inner tetrahedral core
$Core_3(7)$.
The maximal-thickness locus $M_7$ is indicated by black outlines; its cardinality is
recorded in Table~\ref{tab:max-thickness-loci}.}
\label{fig:atlas-n7}
\end{figure}

\begin{figure}[ht]
\centering
\includegraphics[width=.78\textwidth]{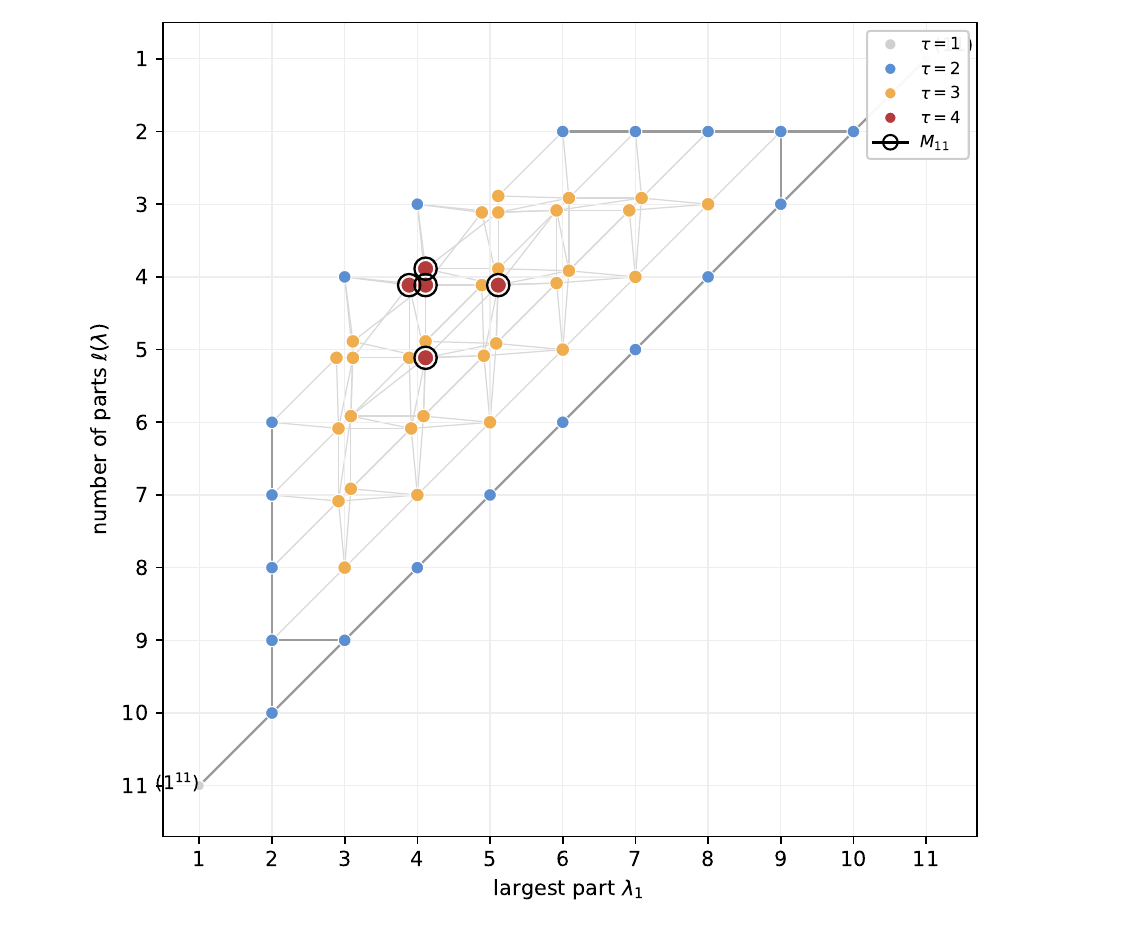}
\caption{Thickness map of $G_{11}$ at the first threshold $n_4=11$ for simplicial
thickness $4$.
The order-$4$ regime is outlined in black and is embedded inside the broader triangular
and tetrahedral layers. Its cardinality is recorded in Table~\ref{tab:max-thickness-loci}.}
\label{fig:atlas-n11}
\end{figure}

\begin{figure}[ht]
\centering
\includegraphics[width=\textwidth]{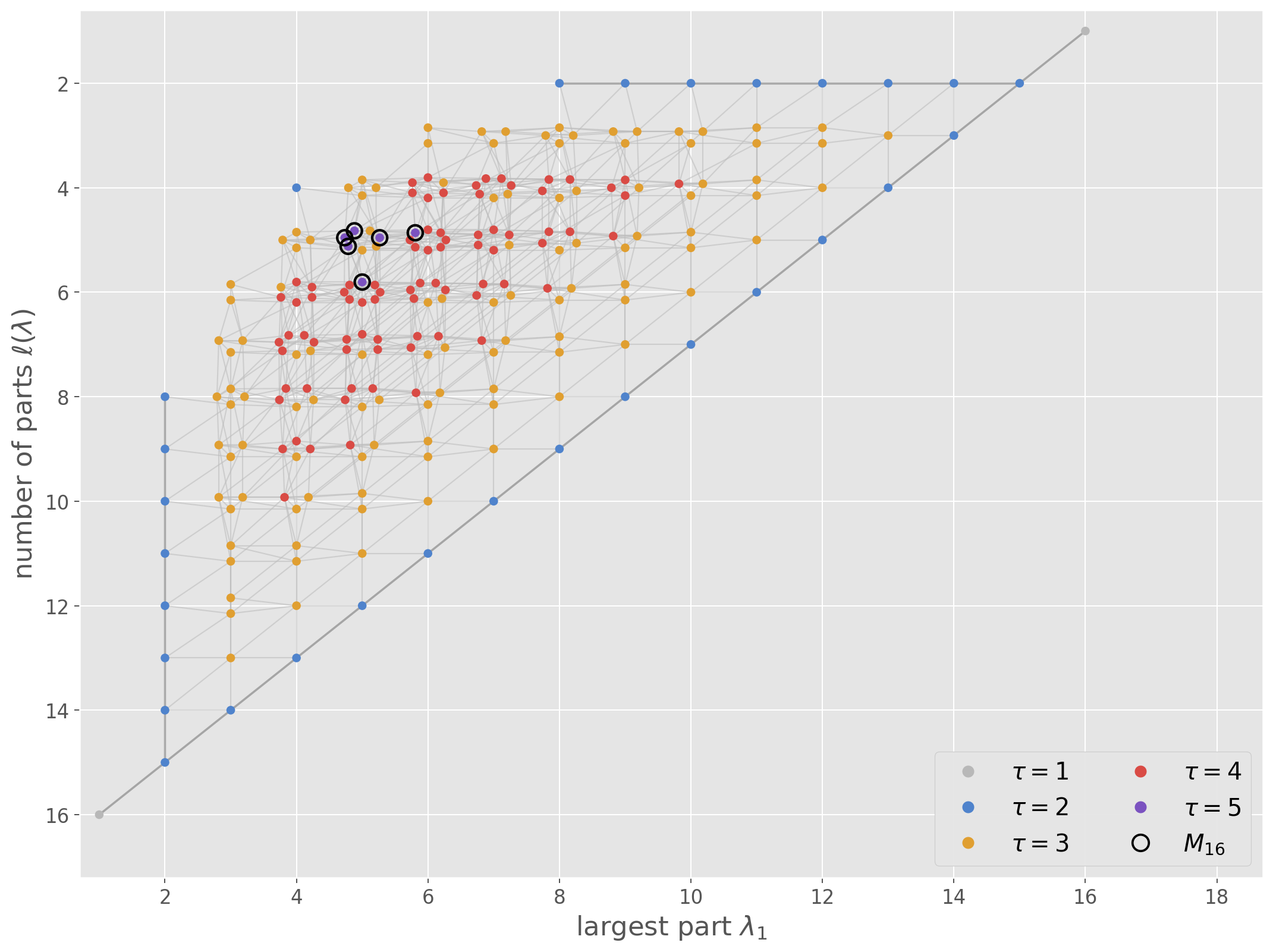}
\caption{Thickness map of $G_{16}$ at the first threshold $n_5=16$ for simplicial
thickness $5$.
The maximal-thickness locus is outlined in black and remains localized well inside the
triangular skin and away from the front extremes of the framework. Its cardinality is
recorded in Table~\ref{tab:max-thickness-loci}.}
\label{fig:atlas-n16}
\end{figure}

\begin{figure}[ht]
\centering
\includegraphics[width=\textwidth]{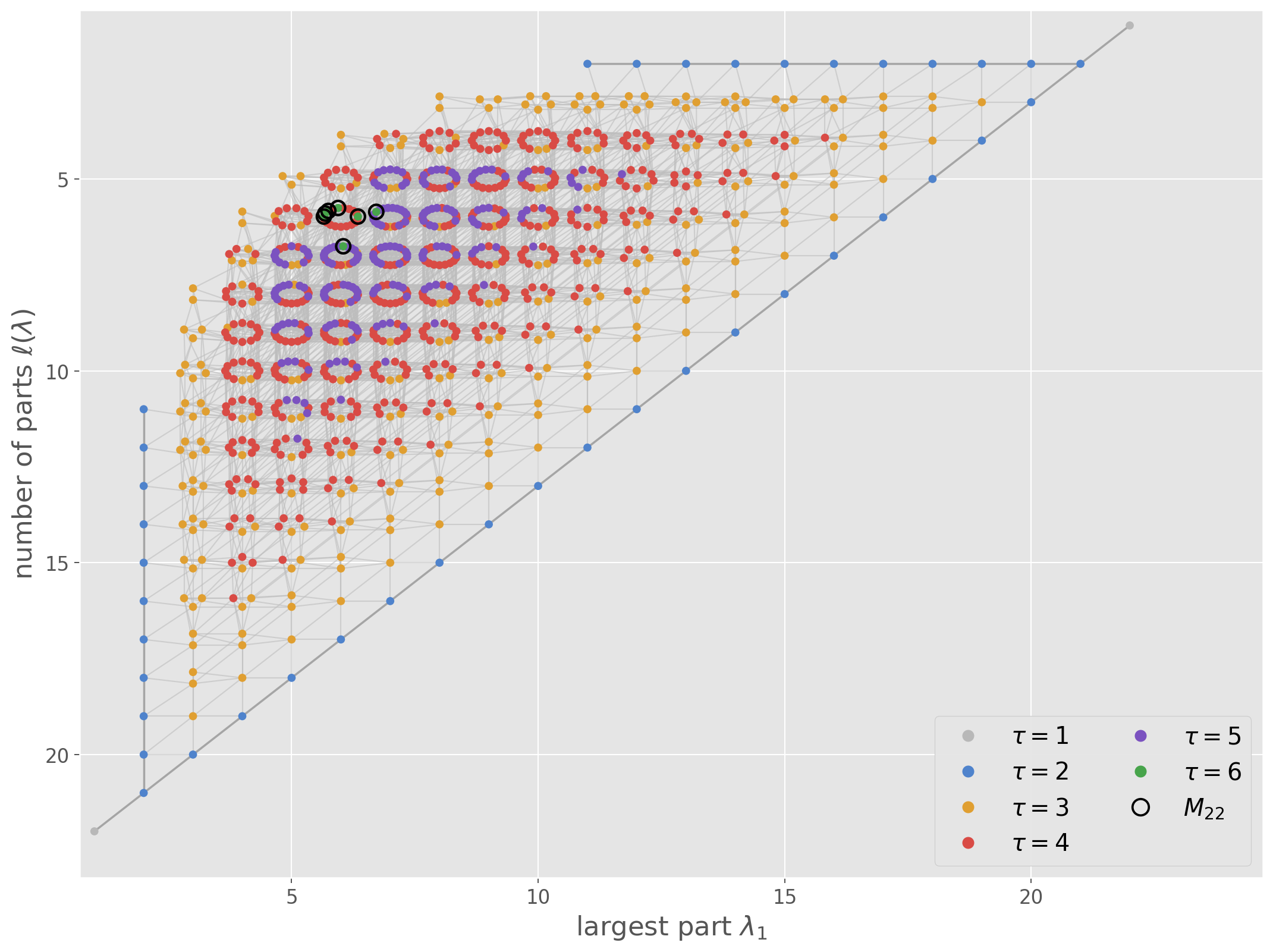}
\caption{Thickness map of $G_{22}$ at the first threshold $n_6=22$ for simplicial
thickness $6$.
The maximal-thickness locus $M_{22}$ is outlined in black. Its cardinality is recorded in
Table~\ref{tab:max-thickness-loci}. Its position is consistent with the same
rear-central bias observed at earlier transition levels.}
\label{fig:atlas-n22}
\end{figure}

\begin{figure}[ht]
\centering
\includegraphics[width=\textwidth]{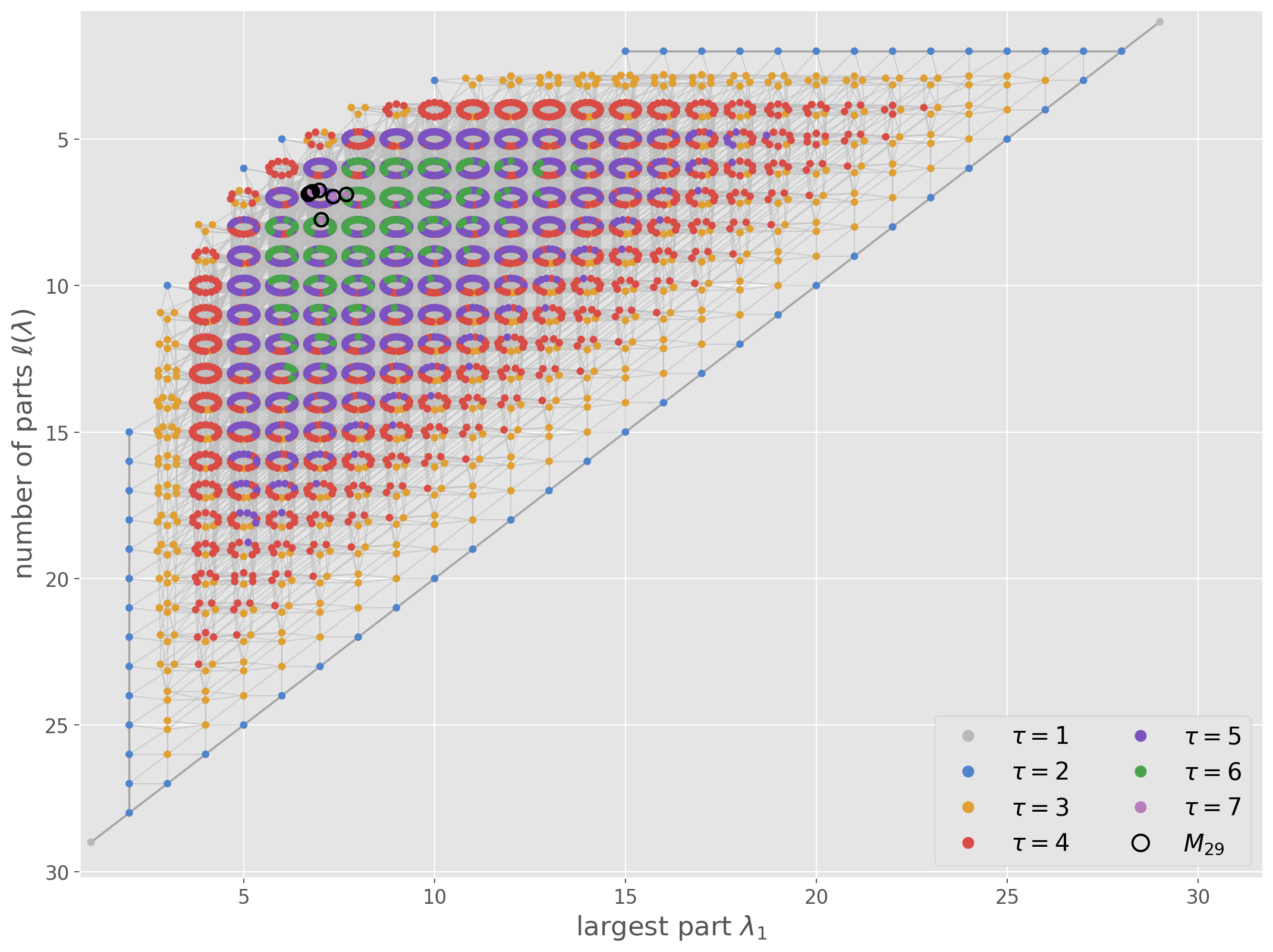}
\caption{Thickness map of $G_{29}$ at the first threshold $n_7=29$ for simplicial
thickness $7$.
The maximal-thickness locus $M_{29}$ is outlined in black. Its cardinality is recorded in
Table~\ref{tab:max-thickness-loci}. Its position continues the same rear-central
localization pattern at the next transition level.}
\label{fig:atlas-n29}
\end{figure}

\begin{remark}
The six figures above show the first triangular, tetrahedral, and order-$4$, order-$5$,
order-$6$, and order-$7$ thresholds.
Their role is to mark the decisive transitions in the growth of simplicial thickness,
rather than to provide an exhaustive pictorial census.
\end{remark}

\subsection{Maximal-thickness loci}

We next record the loci
\[
M_n=\{\lambda:\tau_n(\lambda)=\tau_{\max}(n)\}
\]
of vertices of maximal simplicial thickness.

\begin{table}[ht]
\centering
\begin{tabular}{c|c|c|p{6.2cm}|c}
$n$ & $\tau_{\max}(n)$ & $|M_n|$ & representative element(s) of $M_n$ & location \\
\hline
7  & 3 & 4 & $(4,2,1)$; $(3,3,1)$ & rear-central \\
11 & 4 & 5 & $(5,3,2,1)$; $(4,4,2,1)$ & rear-central \\
16 & 5 & 6 & $(6,4,3,2,1)$; $(5,5,3,2,1)$ & rear-central \\
22 & 6 & 7 & $(7,5,4,3,2,1)$; $(6,6,4,3,2,1)$ & rear-central \\
29 & 7 & 8 & $(8,6,5,4,3,2,1)$; $(7,7,5,4,3,2,1)$ & rear-central
\end{tabular}
\caption{Selected maximal simplicial thickness data at the transition values where a new highest regime first appears.}
\label{tab:max-thickness-loci}
\end{table}

\begin{remark}
The geometric labels in Table~\ref{tab:max-thickness-loci} are partly interpretive and
should be read together with Figures~\ref{fig:atlas-n7}, \ref{fig:atlas-n11}, \ref{fig:atlas-n16}, \ref{fig:atlas-n22}, and~\ref{fig:atlas-n29}.
In particular, the designation ``rear-central'' is used here as a descriptive layout label
rather than as a formally axiomatized subset of $G_n$.
\end{remark}

The conjugation symmetry proved earlier implies that each $M_n$ is conjugation-invariant.
Table~\ref{tab:max-thickness-loci} records the transition values at which a new maximal
simplicial-thickness level first appears.
The atlas shows that this symmetry is compatible with a pronounced geometric bias:
the maximal-thickness loci are not distributed uniformly across the graph.

\subsection{Boundary-attached shells versus interior cores}

The atlas also allows one to compare the outer shells $Sh_r(n)$ with the inner cores
$Core_r(n)$ for the first few nontrivial orders.

At the triangular level, the dominant visible phenomenon is boundary-attached: the shell
of order $2$ is the triangular skin $Sh_2(n)$.
At the tetrahedral level and above, one begins to see a different behaviour:
the first substantial higher-dimensional zones are no longer well described simply as a
uniform thickening of the boundary, but instead appear as localized interior or
rear-central concentrations.

\begin{remark}[Computational shell/core atlas pattern]
\label{rem:shell-core-atlas-pattern}
In the computed range, the passage from $Sh_2(n)$ to $T_{\ge 3}(n)$ marks the transition
from an outer two-dimensional skin to a more localized higher-dimensional body.
In particular, tetrahedral and higher-order behaviour is more concentrated and less purely
boundary-driven than the triangular regime.
\end{remark}

\subsection{What the atlas does and does not prove}

The computational atlas establishes finite facts and stable finite-range geometric patterns.
It proves the explicit first-occurrence table in the computed range and supports the
finite-range version of rear-central thickening expressed in
Observation~\ref{obs:rear-central-pattern}.

At the same time, the atlas does \emph{not} by itself prove stronger asymptotic claims.
In particular, it does not establish any of the following:
\begin{itemize}[leftmargin=1.5em]
\item eventual connectedness or uniqueness of $Core_3(n)$;
\item eventual monotonic growth of the cores $Core_r(n)$;
\item confinement of maximal-thickness loci to a uniformly bounded neighborhood of the spine;
\item a universal asymptotic law for the geometry of $T_{\ge r}(n)$.
\end{itemize}

These remain open problems, even if the computed examples strongly suggest some of them.

\subsection{Interpretation}

The atlas confirms the central geometric picture of the paper.

The first shell order at which a nontrivial shell occurs in the computed range is triangular and boundary-attached.
This is the outer skin of the partition graph in the simplicial sense.
Stronger thickening appears later, first at the tetrahedral level and then at higher
orders.
When it appears, as in Figures~\ref{fig:atlas-n7}, \ref{fig:atlas-n11}, \ref{fig:atlas-n16}, \ref{fig:atlas-n22}, and~\ref{fig:atlas-n29}, it is not
front-antennal but rear-central.

Thus the graph does not merely accumulate simplices in an undifferentiated way.
Rather, it develops a stratified body:
first an outer triangular skin, then a more localized tetrahedral and higher-dimensional
interior.
This is the sense in which the shell language introduced in the paper captures the
spatial organization of simplicial thickness in the partition graph.

\clearpage

\section{Conclusion and open problems}

In this paper we introduced a first systematic language for the simplicial body geometry of
the partition graph $G_n$.

The starting point was the local simplex dimension
\[
\tau_n(\lambda)=\dim_{\mathrm{loc}}(\lambda),
\]
viewed here as a simplicial thickness function on the vertices of $G_n$.
From it we obtained the threshold thick zones
\[
T_{\ge r}(n)=\{\lambda:\tau_n(\lambda)\ge r\},
\]
and then refined this formal filtration by separating each threshold zone into its
boundary-attached part $Sh_r(n)$ and its complementary interior part $Core_r(n)$.

This language makes it possible to distinguish several phenomena that should not be
identified with one another:
\begin{itemize}[leftmargin=1.5em]
\item the exact simplicial stratification by the values of $\tau_n$;
\item the threshold thick zones $T_{\ge r}(n)$;
\item the outer shells $Sh_r(n)$ attached to the boundary framework;
\item the inner cores $Core_r(n)$ detached from that framework;
\item the broader geometric interpretation of the graph as acquiring an interior body.
\end{itemize}

At the strict structural level, the paper established the following points.
Using earlier local-morphology results, simplicial thickness is treated here as a
local invariant determined by the local transfer structure of a partition.
In the present paper we prove that it is preserved by conjugation, and so are the induced
threshold zones, shells, and cores.
The antennas remain strictly one-dimensional in the simplicial sense and are excluded from
all nontrivial thick zones.
The shell filtration
\[
Sh_2(n)\supseteq Sh_3(n)\supseteq Sh_4(n)\supseteq\cdots
\]
is monotone, whereas no comparable monotonicity is asserted for the inner cores.
These facts provide a conservative combinatorial framework for talking about thickening in
$G_n$.

At the geometric level, the paper identified the triangular skin
\[
Sh_2(n)
\]
as the shell of order $2$, that is, the first shell order at which a nontrivial shell can occur in the simplicial-thickness filtration.
This is the boundary-attached part of the first triangular regime.
Beyond it lies the tetrahedral regime $T_{\ge 3}(n)$ and then the higher simplicial regimes
$T_{\ge r}(n)$, $r\ge 4$, whose emergence is interpreted as the onset of a genuinely
higher-dimensional body.

The computational atlas then adds a further geometric message.
Within the computed range, the first substantial higher-dimensional thickening is
concentrated not at the front extremes of the graph, but in its rear-central part.
This rear-central bias is one of the main qualitative conclusions of the paper, although
its present status remains computational rather than fully asymptotic.

The language introduced here is intended as an initial framework rather than a final
theory.
It isolates the minimal formal structure needed to discuss the body geometry of the
partition graph, while leaving several natural directions open.

\subsection{Open problems}

We conclude by listing some natural problems suggested by the present work.

\begin{problem}[First-occurrence theory]
Determine explicit formulas, bounds, or structural criteria for the first-occurrence values
\[
n_r=\min\{n:T_{\ge r}(n)\neq\varnothing\}.
\]
Even partial results for the first few values $n_2,n_3,n_4$ would already clarify the onset
of simplicial thickening.
\end{problem}

\begin{problem}[Triangular skin geometry]
Give a finer structural description of the triangular skin $Sh_2(n)$.
Is it connected for all sufficiently large $n$?
Does it eventually form a single coherent boundary strip?
How does it interact with the rear contour and the side boundary edges?
\end{problem}

\begin{problem}[Tetrahedral core geometry]
Study the first nontrivial inner cores
\[
Core_3(n), Core_4(n), \dots.
\]
When are they nonempty?
Are they connected?
Do they eventually contain a distinguished main component that deserves to be called the
tetrahedral core of the graph?
\end{problem}

\begin{problem}[Rear-central nucleation]
Turn the rear-central thickening pattern into a precise asymptotic theorem.
This requires a sufficiently robust combinatorial definition of ``rear-central'' that is
compatible with the geometry of $G_n$ and stable across growing values of $n$.
\end{problem}

\begin{problem}[Thickness and axial morphology]
Clarify the relation between high-thickness loci and the self-conjugate axis or the spine.
Do maximal-thickness vertices eventually stay near the axis?
Do higher thick zones organize themselves around a bounded neighborhood of the spine?
\end{problem}

\begin{problem}[Thickness versus boundary morphology]
Study the interaction between simplicial thickness and the previously developed outer
morphology of the graph.
How exactly do the shells $Sh_r(n)$ relate to the framework, to the rear contour, and to
the ear structures?
Can one describe the transition from outer boundary geometry to inner body geometry in a
uniform way?
\end{problem}

\begin{problem}[Monotonicity and persistence]
Determine which aspects of thickening persist monotonically as $n$ grows.
The shell filtration is monotone in the thickness order $r$ for fixed $n$, but little is
currently known about monotonicity in the parameter $n$.
Do the first tetrahedral and higher zones persist under natural overlays
\[
G_n\to G_{n+k}?
\]
\end{problem}

\begin{problem}[Quantitative thickness statistics]
Develop global statistics for simplicial thickness:
the distribution of the values of $\tau_n$,
the size of the threshold zones $T_{\ge r}(n)$,
the size and multiplicity of maximal-thickness loci,
and the asymptotic proportion of vertices belonging to the triangular, tetrahedral, and
higher regimes.
\end{problem}

\begin{problem}[Alternative shell formalisms]
Compare the shell/core formalism used in this paper with other possible approaches.
For example, one may ask whether exact levels $T_{=r}(n)$, distance-based central
neighborhoods, or spine-based neighborhoods yield equivalent or genuinely different notions
of body geometry.
\end{problem}

\subsection{Final perspective}

The partition graph is not merely a graph with many local simplices.
It has a discernible spatial organization of simplicial complexity.
Some parts remain thin, some form an outer triangular skin, and some develop genuinely
higher-dimensional behaviour.
The present paper provides a first conservative language for describing this transition.

In this sense, the paper offers not a final classification but an initial geometric account of the body of $G_n$, aimed at formulating the next questions precisely.
The central contribution of the paper is to formulate the next questions precisely:
where thickening begins, how it propagates, how outer shells pass into inner cores, and why
the strongest thickening appears to be organized rear-centrally rather than frontally.

\section*{Acknowledgements}

The author acknowledges the use of ChatGPT (OpenAI) for discussion, structural planning,
and editorial assistance during the preparation of this manuscript. All mathematical
statements, proofs, computations, and final wording were checked and approved by the
author, who takes full responsibility for the contents of the paper.

\end{document}